\documentclass[12pt]{iopart}
\usepackage{iopams}
\usepackage{algorithm}
\usepackage{algpseudocode}
\usepackage{subfigure}

\expandafter\let\csname equation*\endcsname\relax

\expandafter\let\csname endequation*\endcsname\relax

\usepackage{amsmath}

\usepackage{graphicx} 
\usepackage{amsthm}
\usepackage{tikz}
\newtheorem*{remark}{Remark}
\newtheorem{theorem}{Theorem}
\newtheorem{lemma}[theorem]{Lemma}
\usepackage{xcolor}

\begin{document}

\title[EIT with minimal measurements]{Determination of a Small Elliptical Anomaly in Electrical Impedance Tomography using Minimal Measurements}

\author{Gaoming Chen$^1$, Fadil Santosa$^2$, Aseel Titi$^2$}
\address{$^1$School of Mathematics, School of Mathematics, Renming University, Beijing, China 100872}
\address{$^2$Department of Applied Mathematics and Statistics, Johns Hopkins University, Baltimore, MD 21218, USA}
\ead{fsantos9@jhu.edu}

\vspace{10pt}
\begin{indented}
\item[]March 2024
\end{indented}


\begin{abstract}
    We consider the problem of determining a small elliptical conductivity anomaly in a unit disc from boundary measurements. The conductivity of the anomaly is assumed to be a small perturbation from the constant background. A measurement of voltage across two point-electrodes on the boundary through which a constant current is passed. We further assume the limiting case when the distance between two electrodes go to zero, creating a dipole field. We show that three such measurements suffice to locate the anomaly size and location inside the disc. Two further measurements are needed to obtain the aspect ratio and the orientation of the ellipse. The investigation includes the studies of the stability of the inverse problem and optimal experiment design.
\end{abstract}

\maketitle

\section{Introduction}

This work is motivated by the work of Isakov-Titi \cite{isakov-titi-1, isakov-titi-2} which considers the gravimetry inverse problems from minimal measurements. In inverse gravimetry, the goal is to determine the density of a body from gravitational measurements. In their work, the authors break away from classical gravimetry where data are available over a curve or surface. Instead they consider cases where data are available only at a few points. However, instead of reconstructing density distribution, their work focus on recovering parameters of simple geometrical shapes.

In electrical impedance tomography, the objective is to determine the conductivity distribution of a body from boundary measurements. Often called the Calderon problem \cite{calderon}, data typically consists of measurements of boundary voltage to applied currents. The mathematical problem is to find the conductivity of the medium from the Dirichlet-to-Neumann map. In practice, one can place a finite number of electrodes on the boundary of the body, and thus, only capture a finite sampling of the D-to-N map. Nevertheless, with sufficient number of electrodes, one can solve the discretized version of the problem and estimate the interior conductivity from these measurements. The reader is referred to Borcea \cite{borcea} for a comprehensive review of electrical impedance tomography.

Our interest is to see what information about the unknown conductivity can be recovered from a few measurements on the boundary. So, instead of having many electrodes on the boundary, we envision using a pair of electrodes which are moved to a few locations on the boundary. Furthermore, instead of attempting to obtain the conductivity distribution in the body, we seek to determine the location and shape of a conductivity anomaly. There are previous work in this direction, including determination of polygonal anomalies \cite{friedman-isakov}, cracks \cite{friedman-vogelius}, circular anomalies \cite{Kwon-Yoon-Seo-Woo}, and elliptical anomalies \cite{kara-lesnic}. In all these, it is assumed that measurements are available on the entire boundary. 

In this work, we assume that a dipole measurement can be made. By a dipole source, we mean the limiting boundary current of a source-sink pair when the distance between them goes to zero. For voltage measurement, we mean the limiting voltage between the same electrode. We provide a mathematical description of what we mean by measurement in Section 2. We assume that the anomaly in question is elliptical in shape. In the same section, the reader can find a description of the linearized forward map used in this work. In Section 3, we address the inverse problem of determining the anomaly location and size. We provide a direct method to solve the problem, and address uniqueness and stability. The question of what measurements to make, that is, optimal experiment design, is addressed in Section 4. We describe a Bayesian approach and a deterministic approach to find the optimal experiment design. We compare these two approaches in numerical simulations. Section 5 is devoted to the case where we attempt to fully characterize the elliptical anomaly. We find that the problem is unstable. Only the location and size of the anomaly can be robustly estimated. The paper concludes with a discussion section.

\section{Linearized forward map}
We begin by describing the model-to-data relationship in this inverse problem. We are interested in determining the geometric properties of a conductivity anomaly $D$ inside a domain $\Omega$ which is a unit circle. We assume that the conductivity of the medium is given by
\[
\gamma(x) = 1 + \gamma \chi_D (x),
\]
where $\gamma\ll 1$ is a small constant, and $\chi_D(x)$ is a characteristic function supported on $D$. We further assume that $D \Subset \Omega$. The voltage potential caused by external sources is denoted by $U(x)$.  Under linearization, $U(x) = U_0(x) + \gamma U_1(x)$, and the components satisfy
\begin{eqnarray}
    \Delta U_0 &= 0 \\ \label{background}
    \Delta U_1 &= - \nabla \cdot \chi_D \nabla U_0. 
\end{eqnarray}
The so-called background potential $U_0(x)$ is caused by applied currents.  The classical Calderon identity states that for a harmonic function $v$, the following holds
\begin{equation}\label{calderon}
\int_{\partial \Omega} U_1 \frac{\partial v}{\partial n} ds = - \int_D \nabla v \cdot \nabla U_0 dx.
\end{equation}

\begin{figure}
\begin{center}
\begin{tikzpicture}
\draw[line width=1pt] (0,0) circle (2);
\draw[->] (0,0) -- (2.5,0);
\draw[->] (0,0) -- (0,2.5);
\draw[dotted] (0,0) -- (1.8794,0.6840);
\draw[dotted] (0,0) -- (0.6840,1.894);
\draw[red,fill=red] (1.8794,0.6840) circle (.5ex);
\draw[blue,fill=blue] (0.6840,1.8794) circle (.5ex);
\draw (2.2,0.7) node[]{$A$};
\draw (0.6,2.2) node[]{$B$};
\draw (1.1,0.25) node[]{$\phi$};
\draw (0.7,0.65) node[]{$\Delta \phi$};
\draw (-0.6,0.7) node[]{$D$};
\draw (2.5,-0.4) node[]{$x_1$};
\draw (-0.4,2.5) node[]{$x_2$};
\filldraw[cyan, opacity = 0.5, rotate=45] (0.2,0.8) ellipse (0.6 and 0.37); 
\end{tikzpicture}
\end{center}
    \caption{In this schematic, for the case of a source-sink pair, $A$ is the current source and $B$ is the current sink. $A$ is located at coordinates $(\cos\phi,\sin\phi)$ and $B$ is at $(\cos(\phi+\Delta \phi),\sin(\phi+\Delta\phi))$. The conductivity anomaly is represented by the ellipse $D$.  The dipole source is the limiting source-sink pair when $\Delta\phi\rightarrow 0$ (modulo scaling). We say that in this case, the dipole is located at $A$.  The inverse problem is to determine the geometric properties of the anomaly from measurements of the tangential derivative of the perturbational voltage at $A$.  To obtain sufficient information, measurements will be done for several angles $\phi_j$.}
    \label{schematic}
\end{figure}
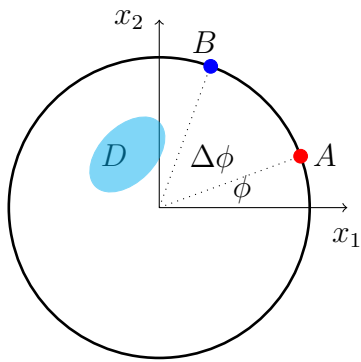

We use polar coordinate $(r,\theta)$; our domain is the unit circle $r\leq 1$. We start by considering a current source-sink pair on the boundary $r=1$ as in Figure \ref{schematic}. A current source is located at $A = (\cos\phi,\sin\phi)$ and a current sink at $B = (\cos(\phi+\Delta \phi),\sin(\phi+\Delta\phi))$. For measurement, the voltage difference between $A$ and $B$ is taken.  In this case, the boundary condition for $U_0$ is
\begin{equation}
    \label{sspair}
    \left. \frac{\partial U_0}{\partial r}\right|_{r=1} = \delta(\theta - \phi) - \delta(\theta - \phi-\Delta\phi).
\end{equation}
In this work, we will focus on dipole sources. This corresponds to the source-sink pair when we divide the right-hand side of (\ref{sspair}) by $\Delta\phi$ and set it to zero.  For a dipole source located at $A = (\cos\phi,\sin\phi)$, the boundary condition for $U_0$ is
\begin{equation}\label{dipole}
\left. \frac{\partial U_0}{\partial r}\right|_{r=1} = -\pi \delta'(\theta - \phi).
\end{equation}
For measurement, the tangential derivative of the voltage is measured at $A$. Thus, when the dipole is at angle $\phi$, the measured data is $\frac{\partial U_1}{\partial\theta}(1,\phi)$.
We set $\int_{\partial \Omega} U_0 ds = 0$ for uniqueness.

The background voltage potential satisfying (\ref{dipole}) is given by \cite{sanvog2}
\begin{equation}
    \label{dipolepot}
    U_0(x_1,x_2) = \frac{-x_1\sin\phi+x_2\cos\phi}{(x_1-\cos\phi)^2+(x_2-\sin\phi)^2}.
\end{equation}
In (\ref{calderon}), we also choose $v=U_0$ so that the left-hand side is the tangential derivative of $U_1$ at the source location, namely,
\[
\frac{1}{2}\frac{\partial U_1}{\partial \theta} (1,\phi),
\]
where the tangential derivative is calculated using polar coordinates. Identifying the left-hand side as data, we identify the forward map, which depends on the measurement angle $\phi$, as
\begin{equation}
    \label{IP0}
    I(D;\phi) = \int_D |\nabla U_0|^2 dx. 
\end{equation}
The inverse problem is to find $D$, properly parameterized, from measurement of the tangential derivative of $U_1$ at $(1,\phi)$ for several $\phi_i$.  To be more precise, we wish to find the $n$ parameters describing $D$ by solving the equation
\begin{equation}
    I(D; \phi_i) = g_i, \;\;\mbox{for} \;\; i=1,2,\dots,n.
\end{equation}

\begin{remark} In practice, one could generate data with a source-sink pair very near each other; i.e., $\Delta \phi$ small in Figure \ref{schematic}. The measured voltage difference can be viewed as impedance between points $A$ and $B$ using Ohm's law. Such resistance could be converted back into voltage difference since $\gamma$ is known.
\end{remark}
We next calculate the integrand in (\ref{IP0}). By direct calculation, we can show that
\begin{equation}\label{kernel2}
|\nabla U_0|^2 =\frac{1}{[(x_1-\cos\phi)^2 + (x_2-\sin\phi)^2]^2 } .  
\end{equation}
The dipole potential field $U_0(x_1,x_2)$ in (\ref{dipolepot}) and its corresponding kernel $|\nabla U_0|^2$ in \eqref{kernel2} are visualized in Figure 2.
\begin{figure}
    \centering
    \includegraphics[width=2.3in]{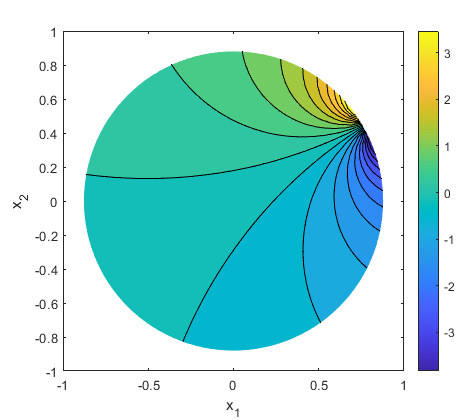}
    \includegraphics[width=2.2in]{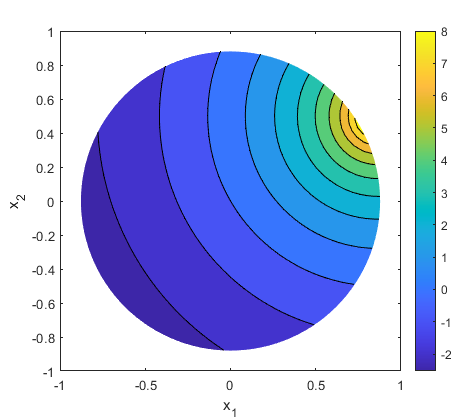}
    \caption{The dipole potential $U_0$ (Left) and the log-plot of the corresponding kernel $|\nabla U_0|^2$ (Right) for a dipole located at $(1,\pi/6)$.}
    \label{fig:pot-kernel}
\end{figure}

For the remainder of this paper, we will assume that $D$ is a small elliptical domain. It is parameterized by the location of the center, $(b_1,b_2)$, the semi-major and semi-minor axes, $(a_1,a_2)$, and the orientation $\xi$. The area of the ellipse is $A=\pi a_1 a_2$ which is small by assumption. To obtain the forward map (\ref{IP0}), we simply integrate $|\nabla U_0|^2$ for a given measurement angle $\phi$ over the ellipse.

We let $K=|\nabla U_0|^2$. For a source located at $(\cos\phi,\sin\phi)$, we have
\begin{equation}\label{Kdef}
K(x_1,x_2) = \frac{1}{[(x_1-\cos\phi)^2+(x_2-\sin\phi)^2]^2}.
\end{equation}
From this, we can calculate
\begin{eqnarray*}
&\frac{\partial K}{\partial x_1} = -\frac{4(x_1-\cos\phi)}{[(x_1-\cos\phi)^2+(x_2-\sin\phi)^2]^3}, \\
&\frac{\partial K}{\partial x_2} = -\frac{4(x_2-\sin\phi)}{[(x_1-\cos\phi)^2+(x_2-\sin\phi)^2]^3},
\end{eqnarray*}
We can calculate second derivatives
\begin{eqnarray*}
K_{11}(x_1,x_2) = & \frac{20(x_1-\cos\phi)^2-4(x_2-\sin\phi)^2}{[(x_1-\cos\phi)^2+(x_2-\sin\phi)^2]^4},\\
K_{12}(x_1,x_2) = & \frac{24(x_1-\cos\phi)(x_2-\sin\phi)}{[(x_1-\cos\phi)^2+(x_2-\sin\phi)^2]^4}, \\
K_{22}(x_1,x_2) = & \frac{-4(x_1-\cos\phi)^2+20(x_2-\sin\phi)^2}{[(x_1-\cos\phi)^2+(x_2-\sin\phi)^2]^4} .
\end{eqnarray*}
In order to integrate $K(x_1,x_2)$ over the ellipse, we use the quadratic approximation of $K(x_1,x_2)$ at $(b_1,b_2)$
\begin{eqnarray} \label{twotaylor}
& \hspace{-2cm} K(x_1,x_2) \approx  K(b_1,b_2) + K_1(b_1,b_2)(x_1-b_1) + K_2(b_1,b_2)(x_2-b_2) \\
& \hspace{-1.5cm} + \frac{1}{2} [ K_{11}(b_1,b_2)(x_1-b_1)^2 +  2 K_{12}(b_1,b_2)(x_1-b_1)(x_2-b_2) + K_{22}(b_1,b_2)(x_2-b_2)^2 ]. \nonumber
\end{eqnarray}

Now, we set up the $(X_1,X_2)$ coordinates so that $X_1$ is aligned with the major axis of the ellipse, and the origin is at $(b_1,b_2)$. Therefore
\begin{eqnarray*}
    x_1-b_1 &= X_1 \cos \xi - X_2 \sin \xi, \\
    x_2-b_2 &= X_1 \sin \xi + X_2 \cos \xi.
\end{eqnarray*}
Let the orthogonal transformation $U$ be given by
\[
U = \left[ \begin{array}{cc}
         \cos\xi & -\sin\xi \\
         \sin\xi & \cos\xi 
         \end{array} \right]  .
\]
Define
\[
M = U^T \left[ \begin{array}{cc}
         K_{11}(b_1,b_2) & K_{12}(b_1,b_2) \\
         K_{12}(b_1,b_2) & K_{22}(b_1,b_2)
         \end{array} \right] U.
\]
Then in (\ref{twotaylor}) becomes
\begin{eqnarray} \label{twotaylor2}
 \hspace{-1cm} K(X_1,X_2) &\approx  K(b_1,b_2) + K_1(b_1,b_2)(X_1\cos\xi-X_2\sin\xi)  \\
& + K_2(b_1,b_2)(X_1\sin\xi+X_2\cos\xi)
 + \frac{1}{2}\left[ M_{11}X_1^2 + 2 M_{12}X_1 X_2 + M_{22}X_2^2 \right]. \nonumber
\end{eqnarray}
We find 
\begin{eqnarray*}
    M_{11} &= K_{11} \cos^2\xi + 2 K_{12} \cos\xi\sin\xi + K_{22} \sin^2\xi, \\
    M_{22} &= K_{11} \sin^2\xi - 2 K_{12} \cos\xi\sin\xi + K_{22} \cos^2\xi.
\end{eqnarray*}

Integrating over the ellipse, we get
\begin{equation}\label{paramtodata}
I(b_1,b_2,a_1,a_2,\xi;\phi) = K(b_1,b_2)(\pi a_1 a_2) 
+ M_{11}\frac{\pi}{8}a_1^3 a_2 + M_{22}\frac{\pi}{8}a_1 a_2^3.
\end{equation}
One can check that in the case of a circle, i.e., $a_1=a_2$, the last two terms become
\[ 
\frac{\pi}{8}a_1^4 (M_{11}+M_{22})=\frac{\pi}{8}a_1^4 \left( (K_{xx}(b_1,b_2)+K_{yy}(b_1,b_2) \right),
\]
and is independent of the orientation $\xi$.

We can separate the terms that depend on the center $(b_1,b_2)$ and those that involve the ellipse $t_1=a_1(\cos\xi,\sin\xi)^T$ and $t_2=a_2(-\sin\xi,\cos\xi)^T$.  We have
\begin{equation} \label{paramtodata2}
I(b_1,b_2,a_1,a_2,\xi;\phi) = \pi a_1 a_2  \left[ 
    K(b_1,b_2;\phi)  + \frac{1}{8} t_1^T \mathcal{K}(b_1,b_2;\phi) t_1 
                    +  \frac{1}{8} t_2^T \mathcal{K}(b_1,b_2;\phi) t_2 \right] ,
\end{equation}
where
\[
\mathcal{K}(b_1,b_2;\phi) = \left[ \begin{array}{cc}
         K_{xx}(b_1,b_2;\phi) & K_{xy}(b_1,b_2;\phi) \\
         K_{xy}(b_1,b_2;\phi) & K_{yy}(b_1,b_2;\phi)
         \end{array} \right]  .
\]

Having derived the forward map, we can now pose the inverse problem. We seek to find $(b_1,b_2)$, $(a_1,a_2)$, and $\xi$, from the equations
\begin{equation}
    I(b_1,b_2,a_1,a_2,\xi;\phi_i) = g_i, \;\; \mbox{for}\;\; i=1,2,\dots,5,
\end{equation}
where the forward map $I$ is given in (\ref{paramtodata}) and $g_i$ are measurements. A simpler problem is one where the goal is to locate the center of the ellipse $(b_1,b_2)$ and its area $A=\pi a_1 a_2$. In this problem, the forward map is obtained after we drop the second order terms in the Taylor's expansion of $K(x_1,x_2)$ (\ref{twotaylor}). Therefore, we have
\begin{equation} \label{IPsimp}
\tilde{I}(b_1,b_2,A;\phi_i) = g_i \;\; \mbox{for}\;\; i=1,2,3,
\end{equation}
where $\tilde{I}(b_1,b_2,A;\phi)=A K(b_1,b_2,\phi)$. We will consider both problems.

\subsection{An alternate forward map using polarization tensors}
A different approach to relating anomalies to boundary potential perturbations was first proposed in \cite{cedio}.  It is based on an asymptotic expansion that exploits smallness of the anomaly, rather than linearization around small conductivity contrast.  For the problem at hand, we use the result in \cite{bruhl}, Theorem 2.1.  Writing the perturbed potential as $U_1(x)$ and the background potential as $U_0(x)$, we have
\begin{equation}\label{momenttensor}
    U_1(x) = -\epsilon\gamma \nabla_y N(b,x) \cdot M \nabla U_0(b),
\end{equation}
where as in above, $b=(b_1,b_2)$ locates the center of the ellipse. In (\ref{momenttensor}), $N(b,x)$ is the Neumann function and $U_0(x)$ is the background potential. Here $\epsilon$ is the size of the anomaly. For our purpose, we can set $\epsilon$ to 1 as long as the axes of the ellipse, $(a_1,a_2)$ are small. The geometric properties of the anomaly is encoded in the polarization tensor $M$ for an ellipse is given by
\[
M = (\pi a_1 a_2) Q^T \left[\begin{array}{cc}
     \frac{a_1+a_2}{a_1+(1+\gamma)a_2}& 0  \\
     0 & \frac{a_1+a_2}{a_2+(1+\gamma)a_1}
\end{array} \right] Q ,
\]
where 
\[
Q = \left[ \begin{array}{cc}
     \cos\xi & \sin\xi  \\
     -\sin\xi & \cos\xi  
\end{array} \right],
\]
and $\xi$ is the orientation of the ellipse as before.

The gradient of the Neumann function $N(x,z)$ in (\ref{momenttensor}) the case of the unit circle is given by \cite{bruhl}
\begin{equation}\label{gradneumann}
    \nabla N(b,x) = \frac{1}{\pi} \left[ \begin{array}{c}
         \frac{x_1 - b_1}{(x_1-b_1)^2 + (x_2-b_2)^2}\\
         \frac{x_2 - b_2}{(x_1-b_1)^2 + (x_2-b_2)^2} 
    \end{array} \right]
\end{equation}
where $|x|=1$. We wish to evaluate the tangential derivative of $U_1$ on the boundary. We write $x=(\cos\alpha,\sin\alpha)$ and directly calculate $\partial U_1/\partial \alpha$, which we will evaluate at $\alpha=\phi$.  Noting that only $\nabla N(b,x)$ depend on $\alpha$, we will work with (\ref{gradneumann}). However, by direct calculation 
\begin{equation*}
\left. \frac{\partial \nabla N(b,x)}{\partial \alpha} \right|_ {\alpha=\phi} =
\frac{1}{\pi} \left[ \begin{array}{c}
      \frac{(b_2^2-b_1^2+1)\sin \phi+2b_2(b_1\cos\phi-1)}{[(\cos\phi-b_1)^2+(\sin \phi-b_2)^2]^2}\\
      \frac{(b_1^2-b_2^2+1)\cos\phi+2b_1(b_2\sin\phi-1)}{[(\cos\phi-b_1)^2+(\sin\phi-b_2)^2]^2}
    \end{array} \right],
\end{equation*}
which is the same as $-\frac{1}{\pi}\nabla U_0(b)$. Thus, we have 
\begin{equation*}
    \frac{\partial U_1(1,\phi)}{\partial \alpha} = \frac{\epsilon\gamma}{\pi} (\nabla U_0(b))^{T} \cdot M \nabla U_0(b),
\end{equation*}
which agrees with \eqref{paramtodata}.

\section{Inversion for anomaly location and size}
Suppose we put a dipole at angle $\phi_i$ and measure the corresponding tangential derivative of $U_1$, namely $\frac{\partial U_1}{\partial \theta}(1,\phi_i) = g_i$. We do this for $i=1,2,3$, with $\phi_i$ distinct. The inverse problem is to determine the location $(b_1,b_2)$ and the area $A$ of the anomaly from the measurements. We use the forward map in (\ref{IPsimp}), so we have
\begin{equation} \label{IPSimp2}
    A K(b_1,b_2;\phi_i) = g_i, \;\;\mbox{for}\;\;i=1,2,3.
\end{equation}

\subsection{Problem geometry}
Define
\begin{equation}\label{Sidefine}
    S_i(b_1,b_2) = (b_1-\cos\phi_i)^2 + (b_2-\sin\phi_i)^2. 
\end{equation}
The function $S_i(b_1,b_2)$ is the squared distance from the anomaly center $(b_1,b_2)$ to the dipole located at $(\cos\phi_i,\sin\phi_i)$. We see that
\[
K(b_1,b_2;\phi_i) = \frac{1}{S_i(b_1,b_2)^2}. 
\]
From the measurement, we can calculate
\[
r_1 = \left[ \frac{g_3}{g_1} \right]^{1/2} =
\frac{S_1(b_1,b_2)}{S_3(b_1,b_2)},
\]
using the definition of $K(x_1,x_2)$. Expanding the above expression, we get
\begin{eqnarray*}
&\left[ b_1 - \frac{r_1\cos\phi_3-\cos\phi_1}{r_1-1} \right]^2
+ \left[ b_2 - \frac{r_1\sin\phi_3-\sin\phi_1}{r_1-1} \right]^2 \\
&=
\left[ \frac{r_1\cos\phi_3-\cos\phi_1}{r_1-1} \right]^2 
+ \left[ \frac{r_1\sin\phi_3-\sin\phi_1}{r_1-1} \right]^2 -1 
\end{eqnarray*}
We can identify the center of the circle $C_1$
\begin{eqnarray*}
    c_{1,1} &= \frac{r_1}{r_1-1} \cos\phi_3 - \frac{1}{r_1-1} \cos\phi_1, \\
    c_{1,2} &= \frac{r_1}{r_1-1} \sin\phi_3 - \frac{1}{r_1-1} \sin\phi_1.
\end{eqnarray*}
It can be seen from the above formula that $C_1=(c_{1,1},c_{1,2})$ is on the line connecting $(\cos\phi_1,\sin\phi_1)$ and $(\cos\phi_3,\sin\phi_3)$. Note that $0<r_1<\infty$. For $0<r_1<1$, $(c_{1,1},c_{1,2})$ is outside the unit circle and beyond $(\cos\phi_1,\sin\phi_1)$, and for $1<r_1<\infty$, it is beyond $(\cos\phi_3,\sin\phi_3)$.  Thus, the center cannot lie on the chord connecting the measurement points. The radius of the circle is
\[
\rho_1=\sqrt{c_{1,1}^2+c_{1,2}^2-1}.
\]
From the third measurement $g_2$, we calculate
\[
r_2=\left[ \frac{g_3}{g_2} \right]^{1/2} =
\frac{(b_2-\cos\phi_2)^2+(b_2-\sin\phi_2)^2}{(b_1-\cos\phi_3)^2+(b_2-\sin\phi_3)^2},
\]
This too forms a circle with center at $C_2=(c_{2,1},c_{2,2})$ with 
\begin{eqnarray*}
    c_{2,1} &= \frac{r_2}{r_2-1} \cos\phi_3 - \frac{1}{r_2-1} \cos\phi_2, \\
    c_{2,2} &= \frac{r_2}{r_2-1} \sin\phi_3 - \frac{1}{r_2-1} \sin\phi_2,
\end{eqnarray*}
and radius $\rho_3=\sqrt{c_{2,1}^2+c_{2,2}^2-1}$. The center of the ellipse, $(b_1,b_2)$ is located at the intersection of the circle with center at $C_1$ with radius $\rho_1$ and the circle with center at $C_2$ with radius $\rho_2$ that lies inside the unit circle.

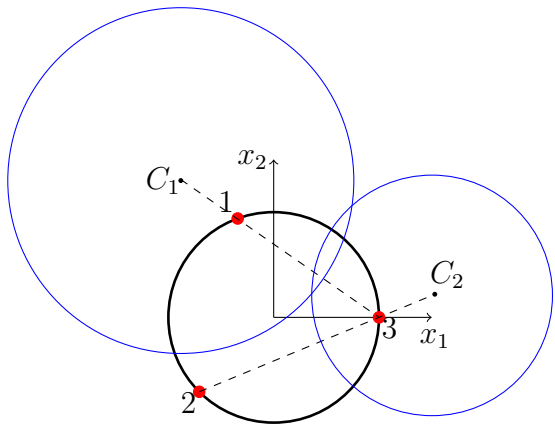
\begin{figure}
    \begin{center}
    \begin{tikzpicture}[scale=1.4]
\draw[line width=1pt] (0,0) circle (1);
\draw[->] (0,0) -- (1.5,0);
\draw[->] (0,0) -- (0,1.5);
\draw[red,fill=red] (1,0) circle (.3ex);
\draw[red,fill=red] (-.707,-.707) circle (.3ex);
\draw[red,fill=red] (-.342,0.9397) circle (.3ex);
\draw[dashed] (-.707,-.707) -- (1.5042,0.2089);
\draw[dashed] (-.882,1.31) -- (1,0);
\draw[blue] (1.5042,0.2089) circle (1.1430);
\draw[blue] (-.882,1.3) circle (1.6430);
\draw (1.53,-0.2) node[]{$x_1$};
\draw (-0.2,1.5) node[]{$x_2$};
\draw (1.1,-.1) node[]{$3$};
\draw (-.807,-.807) node[]{$2$};
\draw (-.45,1.1) node[]{$1$};
\draw[fill=black] (1.53,.22) circle (0.1ex);
\draw (1.65,0.409) node[]{$C_2$};
\draw[fill=black] (-.882,1.3) circle (0.1ex);
\draw (-1.05,1.3) node[]{$C_1$};
\end{tikzpicture}
\end{center}
    \caption{The measurement points $(\cos\phi_1,\sin\phi_1)$, $(\cos\phi_2,\sin\phi_2)$ and $(\cos\phi_3,\sin\phi_3)$ are indicated by $1$, $2$, and $3$. The circle corresponding to the data $r_1$ is centered at $C_1$ with radius is $\rho_1$, and that corresponding to data $r_2$ is centered at $C_2$ with radius $\rho_2$. These are indicated in blue. The intersection of these two circles locate the ellipse center $(b_1,b_2)$.}
    \label{fig:rightangle}
\end{figure}

There are some interesting observations:
\begin{itemize}
    \item When $b_1=0$, $\phi_1=0$ and $\phi_2=\pi$, then $I(\phi_2)\approx I(\phi_1)$ and we get a straight line $b_1=0$. Something similar occurs when $b_1=0$, $\phi_1=\pi/2$ and $\phi_2=-\pi/2$. Then $b_2=0$.
    \item The circles get bigger if the center of the ellipse $(b_1,b_2)$ is away from a measurement point.
    \item The radius grows linearly in the distance from the origin to the center of the ellipse.
\end{itemize}
Once the center of the ellipse is known, one can easily back out the area $A$ from the forward map.

\begin{remark}
    It would be desirable to have a way to characterize the range of the forward map. That is, given a data set, does there corresponds an ellipse whose location and size is consistent with the measured data? We were unable to find an easy characterization. However, it should be pointed that the constructive method described for calculating the center of the ellipse can be viewed as a way to characterize the range. If the data are out of range, the center of the ellipse would land outside the unit circle.
\end{remark}

\subsection{Unique determination of anomaly location and size}
\begin{theorem}
The anomaly location and size is uniquely determined by the data (\ref{IPsimp}) if $\phi_{i}-\phi_{j}\neq 2n\pi$ for $i\neq j$ and $b_{1}^2+b_{2}^2 \leq R < 1$ where $n$ is an integer.  
\end{theorem}
\begin{proof}
Let us assume that we have chosen $\phi_i$ for $i=1,2,3$. The forward map $\mathcal{A}$ takes the parameters $(A,b_1,b_2)$ to measurements according to the left-hand side of (\ref{IPSimp2}). That is 
\begin{equation*}
\mathcal{A}(A,b_1,b_2) =    \left[ \begin{array}{c}
            AK(b_1,b_2,\phi_1)\\
            AK(b_1,b_2,\phi_2)\\
            AK(b_1,b_2,\phi_3)
            \end{array} \right] .
\end{equation*}
The Jacobian of the map can be shown to be
\begin{equation}
	\label{Jacobian3}
	D\mathcal{A} =
 \left[
	\begin{array}{ccc}
		\frac{1}{S_1(b_1,b_2)^2} & -4A\frac{b_1-\cos\phi_1}{S_1(b_1,b_2)^3} & -4A\frac{b_2-\sin\phi_1}{S_1(b_1,b_2)^3} \\
		\frac{1}{S_2(b_1,b_2)^2} & -4A\frac{b_1-\cos\phi_2}{S_2(b_1,b_2)^3} & -4A\frac{b_2-\sin\phi_2}{S_2(b_1,b_2)^3} \\
		\frac{1}{S_3(b_1,b_2)^2} & -4A\frac{b_1-\cos\phi_3}{S_3(b_1,b_2)^3} & -4A\frac{b_2-\sin\phi_3}{S_3(b_1,b_2)^3}
	\end{array} \right] .
\end{equation}
To prove uniqueness we will show that the determinant of the Jacobian is not zero. By direct calculation, we have
\begin{equation}
	\label{DetJacobian3}
\hspace{-0.4cm}	\det(D\mathcal{A})=\frac{64A^2}{(S_1(b_1,b_2)S_2(b_1,b_2)S_3(b_1,b_2))^3} \left|
	\begin{array}{ccc}
		S_1(b_1,b_2) & b_1-\cos\phi_1 & b_2-\sin\phi_1 \\
		S_2(b_1,b_2) & b_1-\cos\phi_2 & b_2-\sin\phi_2 \\
		S_3(b_1,b_2) & b_1-\cos\phi_3 & b_2-\sin\phi_3
	\end{array} \right| .
\end{equation}
Expanding further, we get
\begin{align*}
\negthinspace\det(D\mathcal{A}) & = \frac{64A^2(b_1^2+b_2^2-1)}{(S_1(b_1,b_2)S_2(b_1,b_2)S_3(b_1,b_2))^3}\left[\sin(\phi_1-\phi_2)-\sin(\phi_1-\phi_3)+\sin(\phi_2-\phi_3) \right] \\ 
& = -\frac{256A^2(b_1^2+b_2^2-1)}{(S_1(b_1,b_2)S_2(b_1,b_2)S_3(b_1,b_2))^3}\sin \frac{\phi_1-\phi_2}{2} \sin \frac{\phi_2-\phi_3}{2} \sin \frac{\phi_3-\phi_1}{2},
\end{align*} 
after employing trigonometric identities.  The numerator in the fraction is never zero, and $1-R <S_i(b_1,b_2)<2-R$. Thus, as long as the angles $\phi_i$ are distinct, the determinant never vanishes. 
\end{proof}
Indeed, measurements from any three distinct $\phi_i$ are all that is needed to uniquely determine the location and size of the conductivity anomaly.

\subsection{Stability}

We investigate the stability of determining the center of the anomaly from measurements. For this purpose, we define the forward map $\tilde{\mathcal{A}}$ based on the geometry of the problem
\begin{equation}
\tilde{\mathcal{A}}(b_1,b_2) = 
\left[ 
\begin{array}{c}
\frac{S_1(b_1,b_2)}{S_3(b_1,b_2)} \\
\frac{S_2(b_1,b_2)}{S_3(b_1,b_2)}
\end{array}
\right] .
\end{equation}
The Jacobian matrix is
\begin{equation}
	\label{Jacobian2}
	D\tilde{\mathcal{A}}=
	\begin{pmatrix}
		-\frac{J_1}{S_3^2} & -\frac{J_2}{S_3^2}\\
		-\frac{J_3}{S_3^2} & -\frac{J_4}{S_3^2}
	\end{pmatrix}=
    -\frac{1}{S_3^2}
    \begin{pmatrix}
		J_1 & J_2\\
		J_3 & J_4
	\end{pmatrix}
\end{equation}
where
\begin{align*}
		&J_1:=-2(b_1-\cos\phi_1) S_3(b_1,b_2)+2(b_1-\cos\phi_3) S_1(b_1,b_2),\\ 
		&J_2:=-2(b_2-\sin\phi_1) S_3(b_1,b_2)+2(b_2-\sin\phi_3) S_1(b_1,b_2),\\
		&J_3:=-2(b_1-\cos\phi_2) S_3(b_1,b_2)+2(b_1-\cos\phi_3) S_2(b_1,b_2),\\ 
		&J_4:=-2(b_2-\sin\phi_2) S_3(b_1,b_2)+2(b_2-\sin\phi_3) S_2(b_1,b_2).
\end{align*}
The inverse of the Jacobian matrix is
\begin{equation}\label{Jacobian2inv}
    D \tilde{\mathcal{A}} ^ {-1} = -\frac{S_3^2}{J_1J_4-J_2J_3}
    \begin{pmatrix}
        J_4 & -J_2\\
        -J_3 & J_1
    \end{pmatrix} .
\end{equation}

Next, we examine the denominator in \eqref{Jacobian2inv}. To simplify matters, we set, without loss of generality, $b_2=0$. Then we have
\begin{align*}
		J_1J_4-J_2J_3 = 8b_1(b_1^2-1)[&\sin\phi_1-\sin\phi_2-\sin\phi_3\cos(\phi_1-\phi_3)+\sin\phi_3\cos(\phi_2-\phi_3) \\
        &-\sin(\phi_1-\phi_2)\cos\phi_3)] \\
		+4(b_1^4-1)[&\sin(\phi_1-\phi_2)+\sin(\phi_2-\phi_3)+\sin(\phi_3-\phi_1)].
\end{align*}
Further simplification shows that
\begin{align*}
    J_1J_4-J_2J_3
    &=-16(b_1^2-1)\sin\frac{\phi_1-\phi_2}{2} \sin\frac{\phi_2-\phi_3}{2} \sin\frac{\phi_3-\phi_1}{2}(b_1^2-2b_1\cos\phi_3+1)\\
    &= 16(1-b_1^2)\sin\frac{\phi_1-\phi_2}{2} \sin\frac{\phi_2-\phi_3}{2} \sin\frac{\phi_3-\phi_1}{2}S_3.
\end{align*}



From \eqref{Jacobian2inv}, we have
\[
\| D \tilde{\mathcal{A}}^{-1} \|_1=\frac{1}{|J_1J_4-J_2J_3|} S_3^2 \max\{(|J_4|+|J_3|), (|J_2|+|J_1|)\}.
\]
Recall from the definition of $J_i$ that
\begin{align*}
    |J_4| + |J_3| = |-2(b_2-\sin\phi_2)&S_3 + 2(b_2-\sin\phi_3)S_2| \\
                    &+|-2(b_1-\cos\phi_2)S_3 + 2(b_1-\cos\phi_3)S_2| .
\end{align*}
We can estimate the left-hand side with
\begin{align*}
    |J_4| + |J_3| \leq 2(|b_2-\sin\phi_2| &+ |b_1-\cos\phi_2| )S_3 \\
                    &+ 2(|b_2-\sin\phi_3| + |b_1-\cos\phi_3| )S_2 .
\end{align*}
By Cauchy-Schwarz inequality and the definition of $S_i(b_1,b_2)$
\begin{equation*}
    |J_4| + |J_3| \leq 2\sqrt{2}(S_3\sqrt{S_2} + S_2\sqrt{S_3}).
\end{equation*}
Similarly, $|J_2|+|J_1|\leq 2\sqrt{2}(S_1\sqrt{S_3}+ S_3 \sqrt{S_1})$. These lead to an estimate on the norm of the inverse of the Jacobian
\[
    \|D \tilde{\mathcal{A}}^{-1} \|_1 \leq \frac{2\sqrt{2}S_3^{5/2} }{|J_1 J_4 - J_2 J_3|} \max
    \left\{ S_1+ \sqrt{S_1 S_3}, S_2 + \sqrt{S_2 S_3}\right\} .
\]
Using the formula we obtained for the denominator, we arrive at
\begin{theorem}
With the center of the ellipse at $(b_1,b_2)$ and dipole measurements taken at $\phi_j$, $j=1,2,3$, then the inverse of the Jacobian satisfies
\begin{align}\label{inv_est}
    \| D \tilde{\mathcal{A}}^{-1}\|_1 \leq \frac{2\sqrt{2}}{(1-b_1^2) \left|\sin\frac{\phi_1-\phi_2}{2} \sin\frac{\phi_2-\phi_3}{2} \sin\frac{\phi_3-\phi_1}{2}\right|} & \times \nonumber\\
    & \hspace{-3cm} S_3^{3/2} \; \max \left\{  S_1+ \sqrt{S_1 S_3}, S_2 + \sqrt{S_2 S_3}  \right\} .
\end{align}
\end{theorem}

\begin{remark}
The estimate provides some insights into the stability. There are two terms in \eqref{inv_est}. We start with the second term. Since $S_i(b_1,0)$ is the squared distance from dipole $i$ to the center of the anomaly, we can make this term small by placing the dipoles are close to the anomaly. However, to make the denominator in the first term large, we need to spread out the dipoles. Therefore, the two terms suggest that there is a tradeoff between closeness (measured by small $S_i$) and separateness (measured by the angular differences) of the measurement points.  The factor $S_3^{5/2}$ implies that we must place one electrode close to the anomaly. We will revisit this issue when we examine optimal experiment design in the next section.    
\end{remark}

\section{Optimal experiment design}
\subsection{Bayesian approach}
In (\ref{IPSimp2}), we have discovered the expression of the data with respect to the measurement angle and the parameters, that is
\begin{equation}
	\label{Voltage Expression}
	\begin{aligned}
		&\tilde{I}(A, b_1, b_2; \phi) = AK(b_1,b_2;\phi),\\
		&K(b_1,b_2;\phi) = \frac{1}{((\cos \phi - b_1)^2 + (\sin \phi - b_2)^2)^2}.
	\end{aligned}
\end{equation}
We now consider the problem in the Bayesian statistical framework. The measured data includes the normal noise $\varepsilon$,
\begin{equation*}
	\mathbf{y} = \tilde{I}(\mathbf{A, b_1, b_2}; \phi) + \varepsilon,\ \varepsilon \sim \mathcal{N}(0,\sigma ^2).
\end{equation*}
The prior distributions of the parameters are assumed to be normal
\begin{equation*}
	\begin{aligned}
		&\mathbf{A} \sim \mathcal{N}(\mu_{A}, \sigma_{A}^2),\\
		&\mathbf{b_1}\sim \mathcal{N}(\mu_{b_1}, \sigma_{b_1}^2),\\
		&\mathbf{b_2}\sim \mathcal{N}(\mu_{b_2}, \sigma_{b_2}^2).
	\end{aligned}
\end{equation*}
However, it should be noted that other prior distributions can be used. By taking measurements at $\Phi = (\phi_1, \phi_2, \phi_3)$, three realizations of $\mathbf{y}$ are obtained, denoted as $Y=(y_1, y_2, y_3)$. Our task is to determine measurement locations $\Phi$ whose associated posterior distributions of the parameters provide the most information. In other words, we want to choose $\Phi$ that leads to the least amount of uncertainty in the parameters.

In \cite{HUAN2013288}, a utility function was introduced to quantify the information content about the parameters for a given a design. The utility function uses the Kullback-Leibler divergence, the discrepancy between the prior and the posterior given data $Y$ to measure information gain
\begin{equation*}
	D_{KL}(p_{\mathbf{B|Y};\Phi}||p_{\mathbf{B}|\Phi})=\int_{\mathbb{R}^3} p_{\mathbf{B|Y};\Phi}(B|Y; \Phi)\ln\frac{ p_{\mathbf{B|Y};\Phi}(B|Y; \Phi)}{p_{\mathbf{B}|\Phi}(B|\Phi)}dB.
\end{equation*}
Denote $U(\Phi)$ as the utility of the design $\Phi$. Since the data $Y$ can only be accessed after we have conducted the experiment, we take the expectation of the KL-divergence with respect to $\mathbf{Y}$'s distribution
\begin{equation*}
	\begin{aligned}
		U(\Phi)&=\int_{\mathbb{R}^3}p_{\mathbf{Y}|\Phi}(Y|\Phi)dY\int_{\mathbb{R}^3}p_{\mathbf{B|Y};\Phi}(B|Y; \Phi)\ln\frac{ p_{\mathbf{B|Y};\Phi}(B|Y; \Phi)}{p_{\mathbf{B}|\Phi}(B|\Phi)}dB\\
		&=E \left[\ln\frac{p_{\mathbf{Y|B};\Phi}(\mathbf{Y|B};\Phi)}{p_{\mathbf{Y}|\Phi}(\mathbf{Y}|\Phi)}|\Phi\right] .	\end{aligned}
\end{equation*}
The optimal design is found by maximizing the utility function $U(\Phi)$ over $\Phi$.

In computation, the expectation is approximated with Monte Carlo simulation
\[
U(\Phi)\approx \frac{1}{n_{out}}\sum_{i=1}^{n_{out}}\ln\frac{p_{\mathbf{Y|B};\Phi}(Y_i|B_i;\Phi)}{p_{\mathbf{Y}|\Phi}(Y_i|\Phi)}.
\]
Here $B_i$ are sampled from $\mathbf{B}$ and $Y_i$ are sampled from $\mathbf{Y}|B_i$. The denominator in the logarithm cannot be accessed directly, so an inner simulation is needed.
\[
\begin{aligned}
	p_{\mathbf{Y}|\Phi}(Y|\Phi) 
	& \approx \frac{1}{n_{in}}\sum_{j=1}^{n_{in}} p_{\mathbf{Y|B};\Phi}(Y|B_j;\Phi).
\end{aligned}
\]
Similar as $B_i$, $B_j$ are samples from $\mathbf{B}$. We write the inner and outer simulation together.
\[
U(\Phi)\approx \frac{1}{n_{out}}\sum_{i=1}^{n_{out}}\left\{ \ln p_{\mathbf{Y|B};\Phi}(Y_i|B_i;\Phi)-\ln \left[\frac{1}{n_{in}}\sum_{j=1}^{n_{in}} p_{\mathbf{Y|B};\Phi}(Y_i|B_{ij};\Phi)\right] \right\}.
\]

\subsection{Numerical experiment with Bayesian approach.} 
Here we test the theory with simulations. We set the prior distributions of parameters as
\begin{equation*}
    \begin{aligned}
        &\mathbf{A}\sim\mathcal{N}(0.01, 0.005^2),\\
        &\mathbf{b}_1\sim\mathcal{N}(0.4, 0.2^2),\\
        &\mathbf{b}_2\sim\mathcal{N}(0.3, 0.2^2).
    \end{aligned}
\end{equation*}
The utility function is then maximized on the design space. For simplicity, we restrict our search to `symmetric designs'. That is, with the location $\phi_3$ as center, we constrain $\phi_1$ and $\phi_2$ by setting
\[
\psi := \phi_3 - \phi_1 = \phi_2 - \phi_3.
\]
This reduces the search space $(\phi_3,\psi) \in [0, 2\pi]\times[0,\pi]$. We discretize the space and perform an exhaustive search in Figure \ref{Symmetric Utility}. In practical settings, the optimization problem can be solved numerically using techniques like Bayesian Optimization. The means of the prior for $b_1$, $b_2$, and $A$ are $0.4$, $0.3$, and $0.01$. The utility function has a maximum, and is periodic in $\phi$.

\begin{figure}
    \centering
    \includegraphics[width=3.5in]{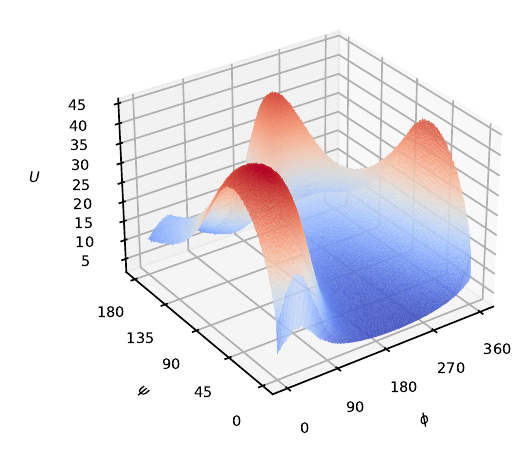}
    \caption{The surface shows the utility function of symmetric designs, where $\phi:=\phi_3$ and $\psi:=\phi_3-\phi_1 = \phi_2-\phi_3$. The maximum value of the utility function $U(\Phi)\approx45.0$ is reached when $\phi\approx36.3^\circ,\psi\approx44.1^\circ$. And the minimum value 2.4 is reached when $\phi\approx218.1^\circ, \psi\approx0.7^\circ$.}
    \label{Symmetric Utility}
\end{figure}

After finding the optimal and the worst experimental designs, we sample parameters from the prior distributions to obtain `ground-truth' parameters: $A=0.0107$, $b_1=0.2760$, $b_2=0.1874$. We simulate the measurement process and produce data from the two designs. We compute the posterior distributions for them. Figures \ref{best-worst} display the posterior distribution for the three parameters. It is clear that the best design produces tighter distributions.
\begin{figure}\label{best-worst}
    \centering
    \setlength{\abovecaptionskip}{10pt}
    \subfigcapskip=-15pt
    \subfigbottomskip=-20pt
    \subfigure[]{\includegraphics[width=2.5in]{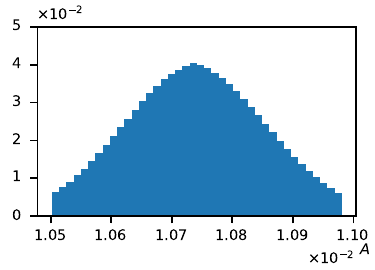}}
    \subfigure[]{\includegraphics[width=2.5in]{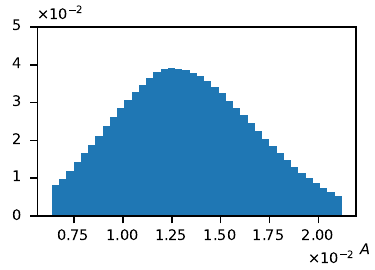}}
    \subfigure[]{\includegraphics[width=2.5in]{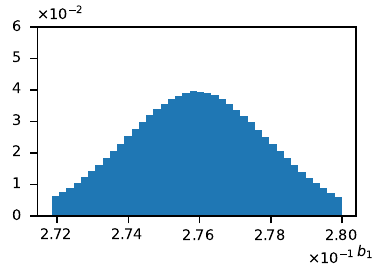}}
    \subfigure[]{\includegraphics[width=2.5in]{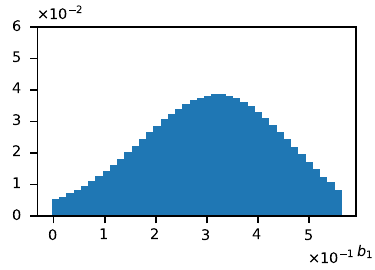}}
    \subfigure[]{\includegraphics[width=2.5in]{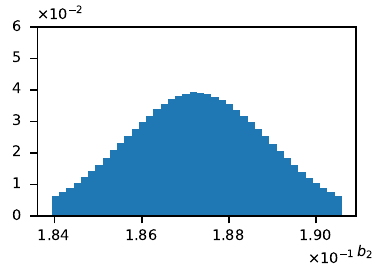}}
    \subfigure[]{\includegraphics[width=2.5in]{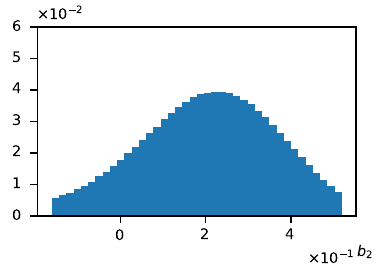}}
    \caption{Posteriors of the parameters for the optimal and the worst experiments obtained by Metropolis-Hastings algorithm. The left column is from the best design when $\phi=36.3^\circ, \psi=44.1^\circ$ and the right column is from the worst design when $\phi=218.1^\circ, \psi=0.7^\circ$. The three rows represent the posterior distribution of the area $A$, the $x$ coordinate $b_1$ and the $y$ coordinate $b_2$. The tightness of the distribution can be clearly seen.}
\end{figure}
For both experiments, we compute the sample mean and 95\% confidence bounds for the parameters, which we summarize below in Table \ref{confidencebounds}.

\begin{table}\label{confidencebounds}
\begin{center}
\begin{tabular}{|c|c|c|c|c|}
\hline\hline
Experiment & parameters & mean     & l.b.      & u.b \\ \hline \hline
           & $A$        & $0.0107$ & $0.0105$  & $0.0110$ \\ \cline{2-5}
Optimal    & $b_1$      & $0.2759$ & $0.2719$  & $0.2800$ \\ \cline{2-5}
           & $b_2$      & $0.1873$ & $0.1839$  & $0.1906$ \\ \hline \hline
           & $A$        & $0.0133$ & $0.0064$  & $0.0212$ \\ \cline{2-5}
Worst      & $b_1$      & $0.3014$ & $-0.0030$ & $0.5635$ \\ \cline{2-5}
           & $b_2$      & $0.2035$ & $-0.1605$ & $0.5198$ \\ \hline \hline 
\end{tabular}
\end{center}
\caption{Comparison of results for the optimal and the worst experiments. Shown are the sample means and 95\% confidence bounds on the three parameters. Ground truth: $A=0.0107$,$b_1=0.2760$, $b_2=0.1874$.}
\end{table}

\subsection{Deterministic approach}

In the deterministic approach for optimal experiment design, we consider the Jacobian in \eqref{Jacobian3}. Optimality is measured by how well-conditioned $D\mathcal{A}$ is for an a priori location $(b_1,b_2)$, and area $A$. To keep the numerical calculation consistent with the above example with the Bayesian approach, we set $b_1=0.4$, $b_2=0.3$, and $A=0.01$. We again consider a symmetric design, so the design space is $\phi_3 \in [0,2\pi]$, and $\psi\in [0,\pi]$ ($\psi=\phi_1-\phi_3$). A proxy for optimality is the reciprocal of the condition number. The larger this number, the more stable the reconstruction around the a prior values.

\begin{figure}
    \centering
    \includegraphics[width=3.5in]{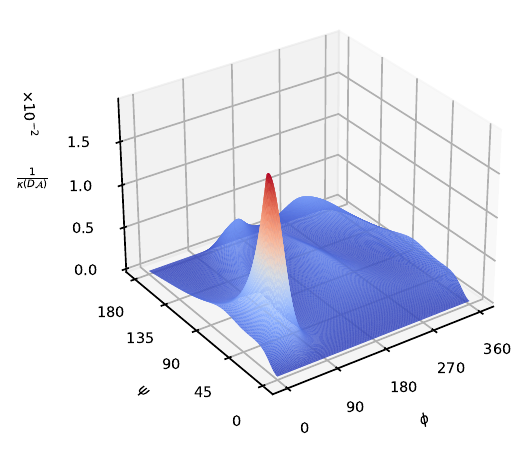}
    \caption{The reciprocal of the condition number as a function of $\phi$ and $\psi$ for $b_1=0.4$, $b_2=0.3$, and $A=0.01$. The maximum is located at $\phi\approx37^\circ$ and $\psi\approx36^\circ$. The surface plot should be compared to that in Figure \ref{Symmetric Utility}.}
    \label{deterministicdesign}
\end{figure}

Figure \ref{Symmetric Utility} and Figure \ref{deterministicdesign} shows some similarity. It should be pointed out that the assumptions in these two approaches are different. In the former, we assume we have priors, and the resulting utility are based on all possible priors. In the latter, the proxy for utility depends on the parameters themselves, which is counter to the nature of inverse problems.  However, one can exploit the easy computation of the utility.  One can start with an experiment design, estimate the parameters and the utility, based on these, update the experiment, i.e., move the electrode locations, until the maximum utility is as large as possible. This idea was explored in \cite{sanvog} in the context of imaging cracks using electrical impedance tomography.

\section{Determination of an elliptical anomaly}

We turn our attention to the problem of finding parameters of a small elliptical anomaly from measured dipole data. The forward map is given in \eqref{paramtodata} and the inverse problem is to find the ellipse center $(b_1,b_2)$, its semi-major and semi-minor axes $(a_1,a_2)$, and its orientation $\xi$ given data $g_i$ in
\[
I(b_1,b_2,a_1,a_2,\xi;\phi_i) = g_i, \;\;\mbox{for}\;\;i=1,2,\dots,5.
\]

\subsection{Jacobian}
We will start by examining the structure of the Jacobian of the forward map. The following calculations will reveal that the problem is highly ill-posed. The findings of this section informs the computational method which we present in the following subsection.

Recall the $I$ takes the form
\begin{equation}\label{system}
I(b_1,b_2,a_1,a_2,\xi;\phi_i) = K(b_1,b_2;\phi_i)(\pi a_1 a_2) 
+ M_{11}\frac{\pi}{8}a_1^3 a_2 + M_{22}\frac{\pi}{8}a_1 a_2^3 .
\end{equation}
We will write $M_{11}=K^2m_1$ and $M_{22}=K^2m_2$, where 
\begin{align*}
m_1=[20(b_1-\cos\phi)^2-4(b_2-\sin\phi)^2]\cos^2\xi &+[48(b_1-\cos\phi)(b_2-\sin\phi)] \cos \xi\sin \xi\\
&+[-4(b_1-\cos\phi)^2+20(b_2-\sin\phi)^2]\sin^2 \xi,
\end{align*}
and 
\begin{align*}
m_2=[20(b_1-\cos\phi)^2-4(b_2-\sin\phi)^2]\sin^2\xi&-[48(b_1-\cos\phi)(b_2-\sin\phi)] \cos\xi\sin\xi\\
&+ [-4(b_1-\cos\phi)^2+20(b_2-\sin\phi)^2]\cos^2\xi.
\end{align*}
Then the partial derivatives of \eqref{system}:
\begin{align*}
\dfrac{\partial I}{\partial b_1} = & \pi a_1 a_2 \frac{\partial K}{\partial b_1} + \frac{\pi}{4} K \frac{\partial K}{\partial b_1} m_1a_1^3a_2 + \frac{\pi}{8}K^2\frac{\partial m_1}{\partial b_1}a_1^3a_2 + \frac{\pi}{4}K\frac{\partial K}{\partial b_1}m_2a_1a_2^3 +\frac{\pi}{8}K^2\frac{\partial m_2}{\partial b_1}a_1a_2^3,\\
\dfrac{\partial I}{\partial b_2} = & \pi a_1 a_2 \frac{\partial K}{\partial b_2} + \frac{\pi}{4} K \frac{\partial K}{\partial b_2} m_1a_1^3a_2 + \frac{\pi}{8}K^2\frac{\partial m_1}{\partial b_2}a_1^3a_2 + \frac{\pi}{4}K\frac{\partial K}{\partial b_2}m_2a_1a_2^3 +\frac{\pi}{8}K^2\frac{\partial m_2}{\partial b_2}a_1a_2^3,\\
\dfrac{\partial I}{\partial a_1}= & \pi a_2K + \frac{3\pi}{8}K^2 m_1a_1^2a_2 + \frac{\pi}{8}K^2m_2a_2^3,\\
\dfrac{\partial I}{\partial a_2}=&\pi a_1K + \frac{\pi}{8}K^2 m_1a_1^3 + \frac{3\pi}{8}K^2m_2a_1a_2^2,\\
\dfrac{\partial I}{\partial \zeta}=&\frac{\pi}{8} K^2\frac{\partial m_1}{\partial \zeta}a_1^3a_2 + \frac{\pi}{8}K^2\frac{\partial m_2}{\partial \zeta}a_1a_2^3,~~~\zeta=\sin \xi .
\end{align*}
We use $A$ as our order parameter. Then $a_1$ and $a_2$ are $O(\sqrt{A})$. We see that to leading order 
\[
\frac{\partial I}{\partial b_1}, \frac{\partial I}{\partial b_2} \sim O(A), \;\; 
\frac{\partial I}{\partial a_1}, \frac{\partial I}{\partial a_2} \sim O(\sqrt{A}), \;\;
\frac{\partial I}{\partial \zeta} \sim O(A^2).
\]
The Jacobian matrix of the forward map will have columns whose magnitudes are 
\[
D\mathcal{A} = \left[ \begin{array}{ccccc}
O(A) & O(A) & O(\sqrt{A}) & O(\sqrt{A}) & O(A^2)
                      \end{array} \right] .
\]
Since $A$ is small, the Jacobian, while it may be invertible, is very ill-conditioned. Therefore, we expect the stability of the inverse problem for the five parameters of the elliptical anomaly to be poor.

\subsection{A Newton's method for anomaly determination}
From Section 3.4, we expect that the problem of determining the location and the size of the ellipse to be relatively well-posed, provided that the experiments are well designed. Therefore, we propose the following algorithm.
\begin{algorithm}
\caption{Newton's method}
\begin{algorithmic}
\State {\mbox Given data for 5 dipole locations: $g_j$, $j=1,\dots,5$}
\State {\mbox Choose 3 measurements to estimate $b_1$, $b_2$, $A$}\;
\State {\mbox Set calculated $(b_1,b_2)$, $a_1=a_2=\sqrt{A/\pi}$, and $\zeta\in [-1,1]$ as initial data}\;
\While{$E > \mbox{tol}$}
   \State $(b_1,b_2,a_1,a_2,\zeta) \gets (b_1,b_2,a_1,a_2,\zeta) - \alpha(D\mathcal{A})^\dagger ({\bf g} - {\bf I}(b_1,b_2,a_1,a_2,\zeta))$
   \State $E = \| {\bf g} - {\bf I}(b_1,b_2,a_1,a_2,\zeta) \|$
\EndWhile
\end{algorithmic}

\end{algorithm}
In the algorithm, since the Jacobian $D\mathcal{A}$ is ill-conditioned, we propose using the pseudo-inverse with a cut-off for the smallest singular value. The parameter $\alpha$ is the step size.  The vectors ${\bf g}$ and ${\bf I}$ are column vectors made up of the measurements at the five dipole locations. 

\subsection{Numerical experiments}
For the purpose of understanding the behavior of the proposed Newton's method for the recovery of the parameters of the elliptical anomaly, we perform some tests with simulated data.  Noise is added to the data to assess the sensitivity of the reconstruction to data errors.

In all the experiments, the dipoles are located at angles $(0,90^o,270^o,180^o,45^o)$. The ground truth values are $b_1=0.4$, $b_2=0.5$, $a_1=0.08$, $a_2=0.04$, and $\xi=45^o$. The forward map is computed using \eqref{system}. For noise, we add a mean-zero normally distributed random number with standard deviation that is a small multiple of $|I(b_1,b_2,a_1,a_2,\xi;\phi_i)|$. Thus, we are assuming that the perturbation from noise is proportional to the signal size. We will report the relative $\ell_1$-norm of the noise.

\begin{table}
    \centering
    \begin{tabular}{|c|c|c|c||c|c|c|c|c|} \hline \hline
    & \multicolumn{3}{|c||}{Step 1} & \multicolumn{5}{|c|}{Step 2}\\ \hline
    \mbox{noise} & $b_1$ & $b_2$ & $A$ & $b_1$ & $b_2$ & $a_1$ & $a_2$ & $\xi$ \\ \hline
$7.5\times 10^{-4}$ & 0.40 & 0.50 & 0.01 & 0.40 & 0.50 & 0.08 & 0.04 & $46^o$ \\ \hline
$1.5\times 10^{-3}$ & 0.40 & 0.50 & 0.01 & 0.40 & 0.50 & 0.09 & 0.04 & $52^o$ \\ \hline
$3.6\times 10^{-3}$ & 0.39 & 0.49 & 0.01 & 0.40 & 0.49 & 0.11 & 0.03 & $42^o$ \\ \hline
$1.2\times 10^{-2}$ & 0.40 & 0.50 & 0.01 & 0.40 & 0.50 & 0.12 & 0.03 & $15^o$ \\ \hline
$1.7\times 10^{-2}$ & 0.40 & 0.49 & 0.01 & 0.40 & 0.50 & 0.10 & 0.03 & $51^o$ \\ \hline\hline
    \end{tabular}
    \caption{Results from numerical experiments. The ground truth is $b_1=0.4$, $b_2=0.5$, $a_1=0.08$, $a_2=0.04$, and $\xi=45^o$. Step 1 is when data from $(0,90^o,270^o)$ are used to estimate the location and size of the anomaly.  This step is quite stable.  Using the values for $b_1$, $b_2$, and $a_1=a_2=\sqrt{A/\pi}$, we apply Newton's method as described in Algorithm 1. The results are reported in the columns under Step 2. The noise level in $\ell_1$ is reported in the first column.  We note that the noise level is quite low.}
\end{table}
The results of the numerical experiments are presented in Table 2. It is clear that the location and the area of the elliptical anomaly are well determined. Sensitivity in the determination of the parameters of the ellipse is higher, as predicted from the nature of the Jacobian matrix. In the experiments, we put a cutoff in the pseudo-inverse of $D\mathcal{A}$ -- we drop singular vectors corresponding singular values below $10^{-6}$.  With this cutoff, the algorithm may diverge depending on noise level. Higher cutoff will stabilize the algorithm at the cost of increased loss of accuracy.  We did not fully explore the interplay between noise level and the cutoff value.

\section{Discussion}
In this work, we examined electrical impedance tomography with minimal measurements. Of interest is the recovery of a small, elliptical conductivity anomaly with small conductivity contrast against a constant background. We show that while finding the location and size of the anomaly is quite stable, the determination of the axes of the ellipse and its orientation is numerically challenging. 

\begin{figure}
    \centering
    \includegraphics[width=3.5in]{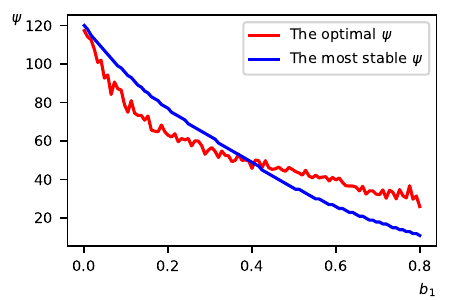}
    \caption{Optimal angle $\psi$ obtained by the Bayesian and deterministic approaches. Details of the numerical experiment are given in Sections 4.2 and 4.3.}
    \label{compare2}
\end{figure}

We also investigated the question of optimal experiment design in the context of Bayesian statistics. We find that if we have a prior on the location and size of the ellipse, we can find the optimal places on the boundary at which measurements should be made. A deterministic approach to optimal experiment design is explored. Here, the outcome depends on the location and size of the ellipse. If we have a prior, then we can use the means as input to find the optimal measurement locations. That the two approaches to optimal experiment design are roughly equivalent can be seen in Figure \ref{compare2}. We set $b_2=0$ and choose $\phi=0$. The only design parameter remaining is $\psi$. Shown in the figure are the angle $\psi$ from the Bayesian approach as a function of the mean $b_1$ in the prior and $\psi$ from the deterministic approach. While the difference is small, what is clear is that optimal experiment design has an important role to play in this problem.

\section*{Acknowledgements} Gaoming Chen's participation in this work was made possible by the Visiting Student Program at Johns Hopkins University. He thanks the many professors and graduate students who made his visit during the Fall of 2023 most enjoyable.

\end{document}